%% file: Amenable_wreath_products_with_non_almost_finite_actions_on_the_Cantor_set.tex
\documentclass[11pt,letterpaper]{amsart} % for a short document

\usepackage{hyperref}
\usepackage{amsmath,mathtools}
\usepackage{amsthm}
\usepackage{amssymb}
\usepackage{dsfont}
\usepackage{ucs}
\usepackage{enumerate}
\usepackage{xcolor}
\usepackage{tikz}
\usepackage{circuitikz}
\usepackage{enumitem}
\usepackage{setspace}
\input{Commandes.tex}

\newtheorem{thm}{Theorem}[section]
\newtheorem*{thm*}{Theorem}
\newtheorem*{classthm*}{Classification Theorem}
\newtheorem{cor}[thm]{Corollary}

\newtheorem{lem}[thm]{Lemma}

\theoremstyle{definition}
\newtheorem{claim}{Claim}
\newtheorem*{claim*}{Claim}

\newenvironment{cproof}{\begin{proof}[Proof of the 
		claim]}{\end{proof}}

\newtheorem{qu}[thm]{Question}

\newtheorem{df}[thm]{Definition}

\title[Amen. wreath prod. with non almost finite actions]{Amenable wreath products with non almost finite actions of mean dimension zero}
\author{Matthieu Joseph}
\address{Université Paris-Saclay, CNRS, Laboratoire de mathématiques d’Orsay, 91405, Orsay}
\email{matthieu.joseph@universite-paris-saclay.fr}


\begin{document}

\maketitle

\begin{abstract}
    Almost finiteness was introduced in the seminal work of Kerr as an dynamical analogue of $\mathcal{Z}$-stability in the Toms-Winter conjecture. In this article, we provide the first examples of  minimal, topologically free actions of amenable groups that have mean dimension zero but are not almost finite. More precisely, we prove that there exists an infinite family of amenable wreath products that admit topologically free, minimal profinite actions on the Cantor space which fail to be almost finite. Furthermore, these actions have dynamical comparison. This intriguing new phenomenon shows that Kerr's dynamical analogue of Toms-Winter conjecture fails for minimal, topologically free actions of amenable groups.
	
	The notion of allosteric group holds a significant position in our study. A group is allosteric if it admits a minimal action on a compact space with an invariant ergodic probability measure that is topologically free but not essentially free. We study allostery of wreath products and provide the first examples of allosteric amenable groups.
\end{abstract}

\maketitle

\tableofcontents

\section{Introduction}

Topological dynamics and $\mathrm{C}^*$-algebra theory have a long intertwined story with the former offering a wide source of meaningful examples for the latter via the notion of crossed product. In the last decade, a major breakthrough occurred in the ambitious classification program for simple separable nuclear $\mathrm{C}^*$-algebras that was launched by Elliott in the $1980$s. This program, which proposes to classify a certain class of $\mathrm{C}^*$-algebras using $K$-theoretic and tracial data has somehow come to a conclusion with the following classification theorem built on the work of many researchers.

\begin{classthm*}[\protect{\cite[Cor.~D]{classification}}]  Unital, separable, simple, nuclear, $\mathcal{Z}$-stable $\mathrm{C}^*$-algebras satisfying the Universal Coefficient Theorem are classified by their Elliott invariant.
\end{classthm*}
It has naturally become an essential task to find large classes of $\mathrm{C}^*$-algebras that are covered by the classification theorem. Crossed products arising from topological dynamics are typical candidates: if $\Gamma\curvearrowright X$ is a continuous action on a compact metric space, then the reduced crossed product $C(X)\rtimes_r \Gamma$ is unital and separable, is simple if and only if $\Gamma\curvearrowright X$ is topologically free and minimal \cite{ArchboldSpielberg}, is nuclear if and only if the action $\Gamma\curvearrowright X$ is amenable \cite{Anantharaman}. In the case where $\Gamma$ is amenable, then the action $\Gamma\curvearrowright X$ is automatically amenable. Moreover, if $\Gamma\curvearrowright X$ is amenable, then $C(X)\rtimes_r\Gamma$ satisfies the Universal Coefficient Theorem \cite[Prop.~10.7]{Tu}. In order to understand whether such crossed products are covered by the classification theorem, it remains to understand their $\mathcal{Z}$-stability.

Even though $\mathcal{Z}$-stability for an amenable action of a non-amenable group is a stimulating question which has received attention recently (see for instance \cite{PICP} and \cite{PICP2}) we focus in the present paper on actions of amenable groups. In this context, several techniques with a dynamical flavor have been developped recently to prove $\mathcal{Z}$-stability. % Ici parler de mean dimension zero. 
 
Mean dimension (as defined by Gromov) appears to be a meaningful invariant in this context. This invariant, associated to any topological action, measures how the dimension grows asymptotically and is known to be zero when the space is finite dimensional or when the action has zero entropy. For a minimal action of $\Z$ on a compact metric space $X$, mean dimension zero implies that $C(X)\rtimes\Z$ is $\mathcal{Z}$-stable, as a combination of \cite{TW} when $X$ is finite dimensional and \cite{ElliottNiu} when $X$ is infinite dimensional. Conversely, Giol and Kerr exhibit examples of minimal $\Z$-actions of mean dimension nonzero whose associated crossed product is not $\mathcal{Z}$-stable \cite{GiolKerr}. More recently, Niu proved that for any free minimal action of $\Z^d$ on a compact metric space $X$, mean dimension zero implies $\mathcal{Z}$-stability of $C(X)\rtimes\Z^d$ \cite{Niu}. For minimal, topologically free actions of amenable groups, mean dimension zero is conjectured to be the suitable setting in which $\mathcal{Z}$-stability holds. 

Another fruitful technique for $\mathcal{Z}$-stability is the recent dynamical notion of \emph{almost finiteness} due to Kerr \cite{Kerr}. Almost finiteness is defined for group actions on compact metric spaces as a sort of topological version of the Ornstein-Weiss tower theorem in measurable dynamics. It was first used in \cite{CJKMSTD} to prove that the crossed product associated with any almost finite free minimal action on the Cantor space is $\mathcal{Z}$-stable. Kerr was then able to remove the zero-dimensional assumption on the space: he proved that crossed products associated with free minimal almost finite actions of amenable groups are $\mathcal{Z}$-stable \cite[Thm.~12.4]{Kerr} and therefore are covered by the classification theorem. A large class of group actions are covered by this theorem: by a recent result of Kerr and Naryshkin, any free minimal action of an elementary amenable group on a finite-dimensional compact metric space is almost finite \cite{KerrNaryshkin}, see also \cite{naryshkin2023group} for an even larger class of group actions. The main result of this article is to provide the first examples of minimal, topologically free actions of amenable groups on zero dimensional spaces (and therefore of mean dimension zero) that are not almost finite. 

\begin{thm}\label{thm.nonalmostfinite}
Let $\Lambda$ be a finitely generated torsion free nilpotent group and let $d\geq 1$. Then $\Z^d\wr\Lambda$ admits minimal, topologically free, profinite actions on the Cantor space that are not almost finite.  
\end{thm}
  
Here an action $\Gamma\curvearrowright X$ on a zero-dimensional compact metric space $X$ is \defin{profinite} if the orbits of the action of $\Gamma$ on the set $\mathrm{Clo}(X)$ of clopen subsets of $X$ are all finite. We refer to Section \ref{sec.criterion} for a precise definition of profinite actions. 

The wreath products covered by Theorem \ref{thm.nonalmostfinite} are all elementary amenable, so by contrast, all their free minimal actions on finite-dimensional compact metric spaces are almost finite by Kerr and Naryshkin's result \cite{KerrNaryshkin}. 

As a consequence, the hope to use almost finiteness in order to prove $\mathcal{Z}$-stability of the crossed products associated with the actions of Theorem \ref{thm.nonalmostfinite} is vain.

In his seminal work \cite{Kerr}, Kerr developped a dynamical counterpart of the so called Toms-Winter conjecture for simple separable nuclear $\mathrm{C}^*$-algebras. With a view toward finding dynamical analogues of finite nuclear dimension, $\mathcal{Z}$-stability and strict comparison in the Toms-Winter conjecture, Kerr introduced the following properties for actions of amenable groups on compact spaces: 
\begin{enumerate}[label=(\roman*)]
\item\label{item.1} finite tower dimension,
\item\label{item.2} almost finiteness,
\item\label{item.3} dynamical comparison,
\end{enumerate}
and proved that for a free minimal action of an amenable group $\Gamma$ on a finite dimensional compact metric space $X$, with finitely many invariant ergodic measures, \ref{item.1}$\Rightarrow$\ref{item.2}$\Leftrightarrow$\ref{item.3}, see \cite[Thm.~9.3]{Kerr}. The actions that we construct in Theorem \ref{thm.nonalmostfinite} are profinite. Since profinite minimal actions are uniquely ergodic and have comparison (see Lemma \ref{lem.profinitecomparison}), this shows that Kerr's dynamical version of Toms-Winter conjecture fails for topologically free actions of amenable groups (see Corollary \ref{cor.Kerrfalse}). Therefore, a dynamical analogue of Toms-Winter conjecture remains to be understood for topologically free minimal actions.\\ 

The notion of allostery, which has its roots in subgroup dynamics, will be the main tool in the proof of Theorem \ref{thm.nonalmostfinite}. A \defin{minimal ergodic action} $\Gamma\curvearrowright (X,\mu)$ is an action by homeomorphisms of $\Gamma$ on a compact metric space $X$, which is minimal (every orbit is dense), with an ergodic $\Gamma$-invariant Borel probability measure $\mu$. We say that a minimal ergodic action is: 
\begin{itemize}
\item \defin{topologically free} if the set of points with trivial stabilizer is comeager, that is contains a dense countable intersection of open sets.
\item \defin{essentially free} if the set of points with trivial stabilizer has full measure. 
\end{itemize}
It is a classical result that essential freeness implies
topological freeness for a minimal ergodic action, see for instance \cite[Lem.~2.2]{josephContinuumAllosteric2021}. The study of the converse, which is false in general, will be the main tool in this article and led the author to introduce the following denomination in \cite{josephContinuumAllosteric2021}.

\begin{df}
A minimal ergodic action $\Gamma\curvearrowright (X,\mu)$ is \defin{allosteric} if it is topologically free but not essentially free.
\end{df}

In general, the $\mathcal{Z}$-stability of $C(X)\rtimes\Gamma$, and therefore its classifiability, where $\Gamma\curvearrowright(X,\mu)$ is an allosteric action of an amenable group $\Gamma$, remains an open question. 

\begin{qu} Let $\Gamma$ be an amenable group. Let $\Gamma\curvearrowright (X,\mu)$ be a minimal ergodic action of mean dimension zero. If $\Gamma\curvearrowright (X,\mu)$ is allosteric, is the crossed product $C(X)\rtimes\Gamma$ classifiable by its Elliott invariant? 
\end{qu}

A countable group $\Gamma$ is \defin{allosteric} if it admits an allosteric action. 
The existence of allosteric groups, and more precisely groups which admit allosteric profinite actions, was asked by Grigorchuk, Nekrashevich and Sushchanskii in \cite[Prob.~7.3.3]{GrigorchukNekrashevichSushchanskii}. The first examples of allosteric groups were provided by Bergeron and Gaboriau in \cite{BergeronGaboriau}. They proved that any non-amenable free product of two nontrivial residually finite groups is allosteric. An independent proof of this result for free groups of finite rank was obtained by Abért and Elek in the unpublished paper \cite{AbertElek}. In \cite{AbertElek2}, Abért and Elek proved that the free product of four copies of the cyclic group $C_2$ admits an allosteric action
whose orbit equivalence relation is measure hyperfinite. In \cite{josephContinuumAllosteric2021}, the author proved that the fundamental group of any non-amenable surface group is allosteric, providing the first examples of allosteric groups with one end. In \cite{BergeronGaboriau}, \cite{AbertElek}, \cite{AbertElek2}, and \cite{josephContinuumAllosteric2021}, the allosteric actions obtained are all profinite, which answers positively the question \cite[Prob.~7.3.3]{GrigorchukNekrashevichSushchanskii}. 

As of now, examples of allosteric groups are rare. By contrast, there are plenty of groups that are known to be non-allosteric. This is the case for countable groups $\Gamma$ such that the space $\Sub(\Gamma)$ of subgroups of $\Gamma$ is at most countable \cite[Cor.~2.4]{josephContinuumAllosteric2021}. Examples of groups with only countably many subgroups are finitely generated nilpotent groups, more generally polycyclic groups, extensions of Noetherian groups by groups with only countably many subgroups (e.g.~solvable Baumslag-Solitar groups $\mathrm{BS}(1,n)$) as observed in \cite[Cor.~8.4]{BeckerLubotzkyThom}. Another class of non-allosteric groups is given that the groups whose ergodic invariant random subgroups (that is, the ergodic probability measures on $\Sub(\Gamma)$ that are invariant by conjugation) are all atomic. Among such groups, the most emblematic ones are lattices in simple higher rank Lie groups \cite{StuckZimmer}, but plenty of other classes of groups are known to admit only atomic ergodic invariant random subgroups, see for instance \cite{Bekka}, \cite{Creutz}, \cite{CreutzPeterson}, \cite{DudkoMedynets} or \cite{PetersonThom}. More surprisingly, there exist non-allosteric groups with plenty of ergodic invariant random subgroups. This is the case for instance for the group $\mathrm{FSym}(\N)$ of finitely supported permutations on $\N$. Indeed, one can check using  \cite{Vershik} that any ergodic invariant random subgroup of $\mathrm{FSym}(\N)$ contains the normal subgroup $\mathrm{Alt}(\N)$ in its support. This fact implies that $\mathrm{FSym}(\N)$ is not allosteric. A similar kind of argument applies for weakly branch groups: by the double commutator lemma \cite{Zheng}, any ergodic invariant random subgroup of a weakly branch group contains a non-trivial uniformly recurrent subgroup in its support, and thus weakly branch groups are not allosteric.\\

With a view toward proving Theorem \ref{thm.nonalmostfinite}, we study in Section \ref{sec.criterionwreath} allostery for wreath products. Given two countable groups $\Gamma,\Lambda$, the \defin{wreath product} $\Gamma\wr\Lambda$ is the group 
\[\Gamma\wr\Lambda\coloneqq \Big(\bigoplus_{\Lambda}\Gamma\Big)\rtimes\Lambda,\]
where $\Lambda$ acts on the direct sum $\bigoplus_\Lambda\Gamma$ by shifting the copies of $\Gamma$. Given a group $\Gamma$ and a prime number $p$, we say that $\Gamma$ is \defin{residually $p$-finite} if for every nontrivial element $\gamma\in\Gamma$, there exists a normal subgroup $N\trianglelefteq\Gamma$ such that $\Gamma/N$ is a finite $p$-group and $\gamma\notin N$. 

We prove that wreath products with an abundance of finite $p$-quotients are allosteric.
\begin{thm}\label{thm.allosterywreath} Let $\Lambda$ be a countable group and let $d\in\N^*$. Assume that there exist infinitely many prime numbers $p$ such that $\Lambda$ is a residually $p$-finite group. Then $\Z^d\wr\Lambda$ admits profinite allosteric actions.
\end{thm}

To prove this theorem, we develop a profinite criterion that implies allostery. This criterion, which is explained in Section \ref{sec.criterion}, uses a profinite construction and rely on the existence of a sequence of finite index subgroups which mimic allostery at finite stages. In Section \ref{sec.criterionwreath}, we prove Theorem \ref{thm.allosterywreath} by showing that $\Z^d\wr\Lambda$ satisfies this criterion. In particular, the allosteric actions that we obtain are all profinite.

As a corollary of Theorem \ref{thm.allosterywreath}, we provide the first examples of \emph{amenable} allosteric groups, which answers a question of Ortega and Scarparo \cite[Rem.~2.6]{OrtegaScarparo}.

\begin{cor}\label{cor.allosteryamenable} Let $\Lambda$ be a finitely generated torsion-free nilpotent group and let $d\in\N^*$. Then $\Z^d\wr\Lambda$ is an amenable allosteric group.\end{cor}

\begin{proof}[Proof of Corollary \ref{cor.allosteryamenable}] If $\Lambda$ is finitely generated torsion-free and nilpotent, then $\Lambda$ is residually $p$-finite for every prime $p$ by a result of Gruenberg \cite{Gruenberg}. Therefore $\Z^d\wr\Lambda$ is allosteric and amenable.
\end{proof}
%Again, the allosteric actions that we obtain in Corollary \ref{cor.allosteryamenable} are profinite. The groups under consideration in this corollary all have infinite asymptotic dimension, but -- after a first version of the present paper appeared -- Hirshberg and Wu proved that there exist amenable allosteric groups of finite asymptotic dimension \cite{HirshbergWu}. 

% Changement effectué ici

Again, the allosteric actions that we obtain in Corollary \ref{cor.allosteryamenable} are profinite. The groups under consideration in this corollary all have infinite asymptotic dimension, but, after a first version of this paper appeared, Hirshberg and Wu proved -- using the same ideas and methods developed in the present paper -- the following general result: the wreath product $\Gamma\wr\Lambda$ of any nontrivial residually finite abelian group $\Gamma$ by any countably infinite residually finite group $\Lambda$ admits profinite allosteric actions, see \cite[Cor.~11.3]{HirshbergWu}.

%Again, the allosteric actions that we obtain in Corollary \ref{cor.allosteryamenable} are profinite. The group under consideration in this corollary all have infinite asymptotic dimension and we don't know whether allosteric amenable groups with finite asymptotic dimension exist. \\

The proof of Theorem \ref{thm.nonalmostfinite} will be provided in Section \ref{sec.mainthm} at the end of this paper. Let us sketch its proof here. The key observation in order to prove this result is that allosteric actions of amenable groups cannot be almost finite, as almost finite actions are always essentially free for any invariant probability measure (see Lemma \ref{lem.afnonessentiallyfree} for of proof of this fact for zero-dimensional compact metric spaces). Since the allosteric actions that we construct in Corollary \ref{cor.allosteryamenable} are profinite, this provides minimal, topologically free profinite actions on the Cantor space that are not almost finite.
  
\section{Almost finiteness and dynamical comparison} \label{sec.kerr}

In this section we discuss the notions of almost finiteness and dynamical comparison for zero-dimensional compact metric spaces.

A compact metric space $X$ is \defin{zero-dimensional} if it admits a basis of clopen sets. Almost finiteness was defined by Kerr for group actions on compact metric spaces \cite[Def.~8.2]{Kerr} but we will restrict in this paper to group actions on zero-dimensional compact metric spaces as the definition is easier to state in this case. 

\begin{df}
An action $\Gamma\curvearrowright X$ on a compact metric zero-dimensional space $X$ is \defin{almost finite} if for all finite $K\Subset\Gamma$ (here and thereafter, $\Subset$ stands for ``is a finite subset of'') and $\varepsilon >0$, there exists $V_1,\dots,V_n\subseteq X$ clopen and $S_1,\dots,S_n\Subset \Gamma$ such that 
\begin{itemize}
\item the sets $sV_i$ for $s\in S_i$ and $i\in\{1,\dots,n\}$ are pairwise disjoint,
\item for all $i\in\{1,\dots,n\}$, $\lvert KS_i\triangle S_i\rvert <\varepsilon \lvert S_i\rvert$,
\item $X=\displaystyle\bigsqcup_{i=1}^nS_iV_i$.
\end{itemize}
\end{df}

Observe that the notion of almost finiteness makes sense for action of \emph{amenable} groups, as the finite sets $S_i$ are $(K,\varepsilon)$-F\o{}lner sets of $\Gamma$. If $\Gamma\curvearrowright X$ is an action on a compact space, we denote by $\mathrm{Prob}_\Gamma(X)$ the space of $\Gamma$-invariant Borel probability measures on $X$. 

\begin{lem}\label{lem.afnonessentiallyfree} Let $\Gamma\curvearrowright X$ be an action on a zero-dimensional compact metric space, which is almost finite. Then for any $\mu\in\mathrm{Prob}_{\Gamma}(X)$, the action $\Gamma\curvearrowright (X,\mu)$ is essentially free. 
\end{lem}

\begin{proof}
Fix $\mu\in\mathrm{Prob}_{\Gamma}(X)$. Let $\gamma\in\Gamma\setminus\{e_\Gamma\}$, let $K\coloneqq\{\gamma\inv\}$ and fix $\varepsilon >0$. Let $V_1,\dots,V_n\subseteq X$ clopen and $S_1,\dots,S_n\Subset\Gamma$ finite that witness almost finiteness for the pair $(K,\varepsilon)$. If $\{x\in X\colon \gamma x=x\}$ intersects $sV_i$, then $\gamma sV_i$ and $sV_i$ are not disjoint, so $s\in S_i\setminus KS_i$. Thus
\begin{align*}
\mu(\{x\in X\colon \gamma x=x\})&\leq \sum_{i=1}^n\lvert KS_i\triangle S_i\rvert \mu(V_i) \\ 
&\leq \varepsilon \sum_{i=1}^n\lvert S_i\rvert\mu(V_i)=\varepsilon.
\end{align*}
This implies that $\Gamma\curvearrowright (X,\mu)$ is essentially free. 
\end{proof}

\begin{df}
An action $\Gamma\curvearrowright X$ on a compact metric zero-dimensional space has comparison if for every nonempty clopen sets $A,B\subseteq X$ satisfying $\mu(A)<\mu(B)$ for all $\mu\in\mathrm{Prob}_\Gamma(X)$, there exists a partition $A=\bigsqcup_{i=1}^n C_i$ into clopen sets and elements $\gamma_1,\dots,\gamma_n\in\Gamma$ such that $(\gamma_iC_i)_{1\leq i\leq n}$ are pairwise disjoint subsets of $B$.  
\end{df}

The rest of this section is devoted to the proof that profinite minimal actions have dynamical comparison. Recall that an action $\Gamma\curvearrowright X$ on a zero-dimensional compact metric space $X$ is \defin{profinite} if the orbits of the $\Gamma$-action on the set $\mathrm{Clo}(X)$ of clopen subsets of $X$ are all finite. In the proof of the following lemma, we will need the fact that profinite minimal actions are uniquely ergodic (see Lemma \ref{lem.profinite} for a more precise statement). Let us also provide some preliminaries on Boolean algebras. The powerset $2^X$ together with the operations of union, intersection and complement, is a Boolean algebra and $\mathrm{Clo}(X)$ is a Boolean subalgebra of it.  Given a Boolean algebra $\mathcal{B}$ and a subset $E\subseteq\mathcal{B}$, the intersection of all Boolean subalgebras of $\mathcal{B}$ that contains $E$ is called the Boolean algebra generated by $E$. A Boolean algebra generated by a finite set is necessarily finite (\cite[Ch.~11, Cor.~2]{BooleanAlgebras}). Let $\mathcal{B}$ be a Boolean subalgebra of the powerset $2^X$. An atom of $\mathcal{B}$ is any non-empty element $A\in\mathcal{B}$ such that the only elements of $\mathcal{B}$ contained in $A$ are $\emptyset$ and $A$. If $\mathcal{B}$ is finite, then any element of $\mathcal{B}$ can be written uniquely as a union of atoms of $\mathcal{B}$ (\cite[Ch.~11, Thm.~2]{BooleanAlgebras}). We are now ready to state and prove the following.

\begin{lem}\label{lem.profinitecomparison}
Any profinite minimal action $\Gamma\curvearrowright X$ has comparison. 
\end{lem}

\begin{proof}
Let $\mu$ be the unique $\Gamma$-invariant probability measure on $X$. Fix $A,B$ clopen such that $\mu(A)<\mu(B)$. Since $\Gamma\curvearrowright X$ is profinite, the set $E$ of all $\Gamma$-translates of $A$ and of $B$ is finite. Therefore, the Boolean algebra $\mathcal{B}$ generated by $E$ is finite. Let $\mathcal{A}\subseteq\mathcal{B}$ denotes the set of atoms of $\mathcal{B}$. By definition of an atom, the set $\mathcal{A}$ is $\Gamma$-invariant and two atoms are either disjoint of equal. For all $C\in\mathcal{A}$, we have $\bigcup_{\gamma\in\Gamma}\gamma C=X$ by minimality. This shows that $\Gamma\curvearrowright\mathcal{A}$ is transitive and that $\mathcal{A}$ forms a partition of $X$ whose pieces all have the same $\mu$-measure.  Since $A$ and $B$ are in $\mathcal{B}$, they can be expressed uniquely as a disjoint union of atoms. So there exists $C_1,\dots,C_n, C_1',\dots,C_m'\in\mathcal{A}$ such that $A=\bigsqcup_{i=1}^nC_i$ and $B=\bigsqcup_{j=1}^{m}C_j'$. Since $\mu(A)<\mu(B)$, then $n<m$ and by transitivity of $\Gamma\curvearrowright\mathcal{A}$, there exist $\gamma_1,\dots,\gamma_n\in\Gamma$ such that $(\gamma_iC_i)_{1\leq i\leq n}$ are pairwise disjoint subsets of $B$. \end{proof}

We therefore obtain the following result as a Corollary of Theorem \ref{thm.allosterywreath}.

\begin{cor}\label{cor.Kerrfalse} Let $\Lambda$ be a finitely generated torsion free nilpotent group and let $d\in\N^*$. Then $\Z^d\wr\Lambda$ admits a minimal, topologically free action $\Gamma\curvearrowright X$ on a finite dimensional compact metric space with finitely many invariant ergodic measures, which is not almost finite but have comparison.  
\end{cor}

By contrast, Kerr proved that for any \emph{free} minimal action of an amenable group on a finite-dimensional compact metric space with finitely many invariant ergodic measures, almost finiteness is equivalent to comparison \cite[Thm.~9.3]{Kerr}. 

\section{A profinite criterion for allostery}\label{sec.criterion}

In this section we provide a profinite criterion which implies allostery. This criterion was used without being made explicit in the proof that the fundamental group of any hyperbolic surface group is allosteric \cite{josephContinuumAllosteric2021}. We provide here a precise and explicit criterion. Let us first recall some basic notions on profinite actions. In Section \ref{sec.kerr}, we defined profinite actions as actions $\Gamma\curvearrowright X$ on zero-dimensional compact metric spaces $X$ such that the $\Gamma$-orbit of every clopen set is finite. We provide here an equivalent definition with inverse limit of actions on finite sets.

Let $(I,\leq)$ be a directed \emph{countable} poset. For all $i\in I$, let $\Gamma\curvearrowright X_i$ be an action on a finite set. Assume that for all $i\leq j$, we have a $\Gamma$-equivariant surjective map $f_{ij}:X_j\to X_i$, such that
\begin{itemize}
\item $f_{ii}$ is the identity on $X_i$,
\item $f_{ik}=f_{ij}\circ f_{jk}$ for all $i\leq j\leq k$. 
\end{itemize}
The inverse limit of the finite spaces $X_i$ is the space
\[\underset{i\in I}{\varprojlim}\ \! X_i \coloneqq\left\{(x_i)\in\prod_{i\in I}X_i\colon x_i=f_{ij}(x_j) \text{ for all }i\leq j\right\}.\]
This space is closed, thus compact metrizable, and totally disconnected in the product topology. The diagonal action of $\Gamma$ on $\prod_{i\in I}X_i$ restricts to an action by homeomorphisms of $\Gamma$ on $\varprojlim X_i$. 

If each $X_i$ is endowed with a $\Gamma$-invariant probability measure $\mu_i$ such that $(f_{ij})_*\mu_j=\mu_i$ for all $i\leq j$, we let $\mu$ be the unique Borel probability measure on $\varprojlim X_i$ that projects for every $j\in I$ onto $\mu_j$ via the canonical projection $\pi_j : \varprojlim X_i \to X_j$. The $\Gamma$-action on $\varprojlim X_i$ preserves $\mu$ and is called the \defin{inverse limit} of the p.m.p.\! actions $\Gamma\curvearrowright (X_i,\mu_i)$. A p.m.p.\! action of $\Gamma$ is \defin{profinite} if it is measurably isomorphic to an inverse limit of p.m.p.\! $\Gamma$-actions on finite sets. The following lemma is well-known, see \cite[Prop.~4.1]{Grigorchuk} for a proof.

\begin{lem}\label{lem.profinite} Let $\Gamma\curvearrowright\varprojlim (X_i,\mu_i)$ be the inverse limit of the p.m.p.\! finite actions $\Gamma\curvearrowright (X_i,\mu_i)$ and let $\mu$ denotes the inverse limit of the $\mu_i$. Then
the following are equivalent:
\begin{enumerate}
\item For every $i\in I$, $\Gamma\curvearrowright X_i$ is transitive and $\mu_i$ is the uniform probability measure on $X_i$. 
\item The action $\Gamma\curvearrowright \varprojlim X_i$ is minimal.
\item The p.m.p.\! action $\Gamma\curvearrowright (\varprojlim X_i, \mu)$ is ergodic. 
\item The action $\Gamma\curvearrowright\varprojlim X_i$ is uniquely ergodic. 
\end{enumerate}
\end{lem}

We are now ready to state the allosteric criterion. 

\begin{thm}[Allosteric criterion]\label{thm.allosterycriterion} Let $\Gamma$ be a countable group and let $S\leq\Gamma$ be a nontrivial subgroup. Fix a family $(\varepsilon_\gamma)_{\gamma\in\Gamma\setminus\{e_\Gamma\}}$ of real numbers in $]0,1[$ such that $\prod_{\gamma\in\Gamma\setminus\{e_\Gamma\}}(1-\varepsilon_\gamma)>0$ and a family $(n_\gamma)_{\gamma\in\Gamma\setminus\{e_\Gamma\}}$ of pairwise coprime integers. Assume that for all $\gamma\in\Gamma\setminus\{e_\Gamma\}$, there exists a finite index subgroup $\Gamma_\gamma\leq\Gamma$, whose index is a power of $n_\gamma$, such that \begin{enumerate}
\item\label{item.topfree} $\gamma\notin\Gamma_\gamma$,
\item\label{item.essfree} $\lvert\{q\in\Gamma/\Gamma_\gamma\colon \forall s\in S,  s q=q\}\rvert\geq (1-\varepsilon_\gamma)[\Gamma:\Gamma_\gamma]$.
\end{enumerate}
Then the profinite action 
\[\Gamma\curvearrowright\underset{F\Subset\Gamma\setminus\{e_\Gamma\}}{\varprojlim} \Big(\Gamma/\bigcap_{\gamma\in F}\Gamma_\gamma, \mu_F\Big),\]
where $\mu_F$ is the uniform probability measure on $\Gamma/\bigcap_{\gamma\in F}\Gamma_\gamma$, 
is allosteric. 
\end{thm}

Before we give the proof of this criterion, let us make some remarks. First, if $\Gamma$ satisfies this criterion, then $\Gamma$ is residually finite. In fact, we don't know any example of a non-residually finite allosteric group. In the criterion, the assumption that the $(n_\gamma)_{\gamma\in\Gamma\setminus\{e_\Gamma\}}$ are pairwise coprime is used to obtain a minimal ergodic profinite action. Item \ref{item.topfree} is used to get topological freeness of the profinite action, whereas Item \ref{item.essfree} is used to get non-essential freeness of the profinite action. 

\begin{proof}
For all $F\Subset\Gamma\setminus\{e_\Gamma\}$, we denote by $X_F$ the finite set $\Gamma/\bigcap_{\gamma\in F}\Gamma_\gamma$. We denote by $X$ the inverse limit $\varprojlim X_F$, by $\mu$ the probability measure on $X$ that projects onto $\mu_F$ and by $\alpha$ the profinite action $\Gamma\curvearrowright (X,\mu)$. For all $F\Subset\Gamma\setminus\{e_\Gamma\}$, the action $\Gamma\curvearrowright X_F$ is transitive, therefore we get by Lemma \ref{lem.profinite} that $\alpha$ is a minimal ergodic action. Let us prove that $\alpha$ is topologically free but not essentially free. For $x\in X$, we denote by $\Stab_\alpha(x)\coloneqq\{\gamma\in\Gamma\colon \gamma x=x\}$ the stabilizer of $x$ for the action $\alpha$. 

We start with topological freeness. For all $F\Subset\Gamma\setminus\{e_\Gamma\}$, let $y_F\coloneqq \bigcap_{\gamma\in F}\Gamma_\gamma\in X_F$ and let $y\coloneqq (y_F)\in X$. By assumption, $\gamma\notin\Gamma_\gamma$ for all $\gamma\in\Gamma\setminus\{e_\Gamma\}$ and thus $\Stab_\alpha(y)=\{e_\Gamma\}$. Since $\alpha$ is minimal, this implies that a dense set of points have trivial stabilizer. Moreover, the set $\Stab_{\alpha}^{-1}(\{e_\Gamma\})=\bigcap_{\gamma\in\Gamma\setminus\{e_\Gamma\}}\{x\in X\colon \gamma x\neq x\}$ is always $G_\delta$. Therefore, it is comeager and this shows that $\alpha$ is topologically free. 

We now prove that $\alpha$ is not essentially free. Since the index of $[\Gamma:\Gamma_\gamma]$ is a power of $n_\gamma$ for all $\gamma\in\Gamma$ and the $n_\gamma$'s are pairwise coprime, the action $\Gamma\curvearrowright X_F$ is isomorphic to the diagonal action $\Gamma\curvearrowright \prod_{\gamma\in F}\Gamma/\Gamma_\gamma$ of the left coset actions. Therefore for all $F\Subset\Gamma\setminus\{e_\Gamma\}$,
\begin{align*}
\frac{\lvert\{x\in X_F\colon \forall s\in S, sx=x\}\rvert}{\lvert X_F\rvert} &=\prod_{\gamma\in F}\frac{\lvert\{q\in \Gamma/\Gamma_\gamma\colon \forall s\in S, sq=q\}\rvert}{[\Gamma:\Gamma_\gamma]} \\ 
& \geq \prod_{\gamma\in F}(1-\varepsilon_\gamma).
\end{align*}
By definition of $\mu$, this implies that \[\mu(\{x\in X\colon S\leq \Stab_{\alpha}(x)\})\geq \prod_{\gamma\in\Gamma\setminus\{e_\Gamma\}}(1-\varepsilon_\gamma)>0\]
which shows that $\alpha$ is not essentially free.
\end{proof}

\section{Allostery and wreath products}\label{sec.criterionwreath}

In this section, we apply the allosteric criterion of Theorem \ref{thm.allosterycriterion} to prove that some wreath products are allosteric. In the sequel, we will only work with wreath products of the form $A\wr\Lambda$ with $A$ abelian and we recall the definitions in this case. Let $A$ and $\Lambda$ be two countable groups with $A$ abelian. The \defin{support} of a function $f : \Lambda\to A$ is defined by 
$\supp(f)\coloneqq\{\lambda\in\Lambda\colon f(\lambda)\neq 0\}$. Pointwise addition defines a group operation on $A^\Lambda$. The subgroup of $A^\Lambda$ consisting of functions whose support is finite is denoted by $\bigoplus_{\lambda\in\Lambda}A$. The \defin{wreath product} of $A$ by $\Lambda$, denoted by $A\wr\Lambda$, is the semi-direct product
\[A\wr\Lambda\coloneqq\textstyle\big(\bigoplus_{\lambda\in\Lambda}A\big)\rtimes\Lambda\] where $\Lambda$ acts by shifting the copies of $A$ in $\bigoplus_{\lambda\in\Lambda}A$. In other words, the group multiplication is given by $(f,\gamma)(f',\gamma')\coloneqq (f+ f'(\gamma\inv\cdot),\gamma\gamma')$. 

The following theorem uses the allosteric criterion described in Theorem \ref{thm.allosterycriterion} to prove that wreath products with an abundance of $p$-finite quotients are allosteric. \begin{thm} Let $\Lambda$ be a countable group. Assume that there exist infinitely many prime numbers $p$ such that $\Lambda$ is a residually $p$-finite group. Then for all $d\geq 1$, the wreath product $\Z^d\wr\Lambda$ is allosteric. 
\end{thm}

\begin{proof}
Let $\Gamma\coloneqq \Z^d\wr\Lambda$. Fix $(p_\gamma)_{\gamma\in\Gamma\setminus\{e_\Gamma\}}$ a sequence of pairwise distinct prime numbers, such that for all nontrivial element $\gamma=(g,\delta)$ in $\Gamma$, the group $\Lambda$ is a residually $p_\gamma$-finite group and $g(\supp(g))\cap (p_\gamma\Z)^d=\varnothing$. Fix also a sequence $(\varepsilon_\gamma)_{\gamma\in\Gamma\setminus\{e_\Gamma\}}$ of real numbers in $]0,1[$ such that $\prod_{\gamma\in\Gamma\setminus\{e_\Gamma\}}(1-\varepsilon_\gamma)>0$. 
%Est-ce que $l_n=\lvert\supp(f_n)\rvert$ convient ? Réponse, non parce que $l_n$ doit être une puissance de $p_n$ 

In the remainder of the proof, we fix a nontrivial element $\gamma\coloneqq (g,\delta)$ in $\Gamma$ and we will construct a finite index subgroup $\Gamma_\gamma$ satisfying the assumptions of the allosteric criterion (Theorem \ref{thm.allosterycriterion}). Fix $l\in\N$ such that $l>\lvert\supp(g)\rvert$. Let us fix a finite index normal subgroup $\Lambda_\gamma\trianglelefteq\Lambda$ by considering the following two cases. 
\begin{itemize}
\item If $\delta=e_\Lambda$, then fix $\Lambda_\gamma\trianglelefteq\Lambda$ of finite index such that: 
\begin{itemize}
\item $[\Lambda:\Lambda_\gamma]$ is a power of $p_\gamma$ that satisfies $l<\varepsilon_\gamma[\Lambda:\Lambda_\gamma]$,
\item for all distinct $\lambda,\lambda'\in\supp(g)$, $\lambda\inv\lambda'\notin\Lambda_\gamma$. 
\end{itemize}
\item If $\delta\neq e_\Lambda$, then fix $\Lambda_\gamma\trianglelefteq\Lambda$ of finite index such that: 
\begin{itemize}
\item $[\Lambda:\Lambda_\gamma]$ is a power of $p_\gamma$ that satisfies $l<\varepsilon_\gamma[\Lambda:\Lambda_\gamma]$,
\item for all distinct $\lambda,\lambda'\in\supp(g)$, $\lambda\inv\lambda'\notin\Lambda_\gamma$,
\item $\delta\notin\Lambda_\gamma$. 
\end{itemize}
\end{itemize}
Such a subgroup $\Lambda_\gamma$ exists because $\Lambda$ is a residually $p_\gamma$-finite group. Observe that the second condition implies that any two distinct elements in $\supp(g)$ belong to distinct left cosets of $\Lambda_\gamma$. Fix a subset $E\subseteq\Lambda/\Lambda_\gamma$ of cardinality $l$ such that all the left cosets of $\Lambda_\gamma$ that intersect $\supp(g)$ belong to $E$. This is possible since $l>\lvert\supp(g)\rvert$. Let $A_\gamma$ be the subgroup of $\bigoplus_{\lambda\in\Lambda}\Z^d$ given by 
\[A_\gamma\coloneqq \Big\{f\in\textstyle\bigoplus_{\lambda\in\Lambda}\Z^d\colon \sum_{\lambda\in q}f(\lambda)\in (p_\gamma\Z)^d \text{ for all }q\in E\Big\}.\]
\begin{claim}$A_\gamma$ is $\Lambda_\gamma$-invariant. 
\end{claim}
\begin{cproof}
Let $f\in A_\gamma$ and $\lambda\in\Lambda_\gamma$. Since $\Lambda_\gamma$ is normal in $\Lambda$, then for all $q\in E$, we have $\lambda\inv q=q$. Therefore,
\[\sum_{\lambda'\in q}f(\lambda\inv\lambda') =\sum_{\lambda'\in q}f(\lambda')\in (p_\gamma\Z)^d.\]
This shows that $A_\gamma$ is $\Lambda_\gamma$-invariant.
\end{cproof}

Let us define $\Gamma_\gamma\coloneqq A_\gamma\rtimes\Lambda_\gamma$. This is a finite index subgroup of $\Gamma$. 
  
\begin{claim}\label{claim.index} $[\Gamma:\Gamma_\gamma]=[\Lambda:\Lambda_\gamma][\bigoplus_{\lambda\in\Lambda}\Z^d:A_\gamma]=[\Lambda:\Lambda_\gamma]
(p_\gamma)^{ld}$, which is a power of $p_\gamma$.
\end{claim}

\begin{cproof} Let us denote by $\boldsymbol{0}$ the neutral element of $\bigoplus_{\lambda\in\Lambda}\Z^d$. The proof is in two steps. Firstly, the map given by 
\[\begin{array}{l|rcl}
\phi: & \Lambda/\Lambda_\gamma & \longrightarrow & \Gamma/(\bigoplus_{\lambda\in\Lambda}\Z^d)\rtimes\Lambda_\gamma \\
    & \xi\Lambda_\gamma & \longmapsto & (\boldsymbol{0},\xi)(\bigoplus_{\lambda\in\Lambda}\Z^d)\rtimes\Lambda_\gamma \end{array}\]
    is well defined, as if $\xi_1\Lambda_\gamma=\xi_2\Lambda_\gamma$, then $(\boldsymbol{0},\xi_1)^{-1}(\boldsymbol{0},\xi_2)=(\boldsymbol{0},\xi_1^{-1}\xi_2)$ belongs to $(\bigoplus_{\lambda\in\Lambda}\Z^d)\rtimes\Lambda_\gamma$. It is moreover straightfoward to check that $\phi$ is one-to-one and onto. Therefore $[\Gamma:(\bigoplus_{\lambda\in\Lambda}\Z^d)\rtimes\Lambda_\gamma]=[\Lambda:\Lambda_\gamma]$. Secondly, the map 
\[\begin{array}{l|rcl}
\psi: & (\bigoplus_{\lambda\in\Lambda}\Z^d)/A_\gamma & \longrightarrow & (\bigoplus_{\lambda\in\Lambda}\Z^d)\rtimes\Lambda_\gamma/\Gamma_\gamma \\
    & f+A_\gamma & \longmapsto & (f,e_\Lambda)\Gamma_\gamma \end{array}\] is well defined, as if $f_1+A_\gamma= f_2+ A_\gamma$, then $(f_1,e_\Lambda)^{-1}(f_2,e_\Lambda)=(f_2-f_1,e_\Lambda)$ belongs to $\Gamma_\gamma$. Once again, a straightforward computation shows that $\psi$ is one-to-one and onto. Therefore, $[(\bigoplus_{\lambda\in\Lambda}\Z^d)\rtimes\Lambda_\gamma:\Gamma_\gamma]=[\bigoplus_{\lambda\in\Lambda}\Z^d:A_\gamma]=(p_\gamma)^{ld}$ and this finishes the proof.
\end{cproof}
%\begin{cproof} A straightforward computation shows that the map \[\Gamma/(A_\gamma\rtimes\Lambda_\gamma)\to\big (\bigoplus_\Lambda\Z^d\big)/A_\gamma\times \Lambda/\Lambda_\gamma\] given by $(f,\lambda)(A_\gamma\rtimes\Lambda_\gamma)\mapsto (f+A_\gamma,\lambda\Lambda_\gamma)$ is a well defined bijection, which proves the claim since $[\bigoplus_\Lambda\Z^d:A_\gamma]=(p_\gamma)^{ld}$. \end{cproof}
  
%\[\Gamma_n\coloneqq\Big\{(f,\lambda)\in \Z^d\wr\Lambda\colon \lambda\in \Lambda_n\text{ and }\sum_{\lambda\in\Lambda_n}f(\lambda\gamma_i)\in (p_n\Z)^d\text{ for all }1\leq i\leq l_n\Big\}.\]
%\begin{claim}$\Gamma_n$ is a subgroup of $\Gamma$.\end{claim}
%\begin{cproof}Fix two elements $(f,\lambda)$ and $(f',\lambda')$ in $\Gamma_n$. Then $(f,\gamma)(f',\gamma')=(f+f'(\gamma\inv\cdot),\gamma\gamma')$. We have that $\gamma\gamma'\in\Lambda_n$ and for all $1\leq i\leq l_n$, we have\begin{align*}\sum_{\lambda\in\Lambda_n}(f+f'(\gamma\inv\cdot))(\lambda\gamma_i) &=\sum_{\lambda\in\Lambda_n}f(\lambda\gamma_i)+\sum_{\lambda\in\Lambda_n}f'(\gamma\inv\lambda\gamma_i)\\&=\sum_{\lambda\in\Lambda_n}f(\lambda\gamma_i)+\sum_{\lambda\in\Lambda_n}f'(\lambda\gamma_i),\end{align*}which belongs to $(p_n\Z)^d$. If $(f,\gamma)$ belongs to $\Gamma_n$, then $(f,\gamma)\inv =(-f(\gamma \cdot),\gamma\inv)$ and for all $1\leq i\leq l_n$, \[\sum_{\lambda\in\Lambda_n}f(\gamma\cdot)(\lambda\gamma_i)=\sum_{\lambda\in\Lambda_n}f(\gamma\lambda\gamma_i)=\sum_{\lambda\in\Lambda_n}f(\lambda\gamma_i),\]which belongs to $(p_n\Z)^d$ and finishes the proof. \end{cproof}

\begin{claim} The element $\gamma=(g,\delta)$ doesn't belong to $\Gamma_\gamma$.  
\end{claim}
\begin{cproof} If $\delta\neq e_\Lambda$, then $\delta\notin\Lambda_\gamma$ by construction. Therefore $\gamma\notin\Gamma_\gamma$. If $\delta=e_\Lambda$, then $\supp(g)$ is not empty since $(g,\delta)\neq e_\Gamma$. Let $\lambda\in\supp(g)$ and let $q\in\Lambda/\Lambda_\gamma$ be such that $\lambda\in q$. By construction of $E$, we have $q\in E$. Since any two distinct elements in $\supp(g)$ belong to different left cosets of $\Lambda_\gamma$, we get that \[\sum_{\lambda'\in q}g(\lambda')=g(\lambda)\]
which does not belong to $(p_\gamma\Z)^d$ because  $g(\supp(g))\cap (p_\gamma\Z)^d=\varnothing$. Therefore $\gamma\notin\Gamma_\gamma$. 
\end{cproof}

Let $S$ be the subgroup of $\Z^d\wr\Lambda$ defined by 
\[S\coloneqq\big\{(f,e_\Lambda)\in\Z^d\wr\Lambda\colon \supp(f)\subseteq\{e_\Lambda\}\big\}.\] 

\begin{claim} 
$\lvert\{q\in \Gamma/\Gamma_\gamma\colon \forall s\in S, sq=q\}\rvert\geq (1-\varepsilon_\gamma)[\Gamma:\Gamma_\gamma]$.
\end{claim}

\begin{cproof} Let $\pi: \Gamma/\Gamma_\gamma\onto \Gamma/(\bigoplus_{\lambda\in\Lambda}\Z^d)\rtimes\Lambda_\gamma$ be the quotient map. This is a $\Gamma$-invariant map whose fibers are all of cardinality $[\bigoplus_{\lambda\in\Lambda}\Z^d:A_\gamma]$ by Claim \ref{claim.index}. Moreover, remark that for all $(f,\lambda)\in\Gamma$, we have $\pi((f,\lambda)\Gamma_\gamma)=(\boldsymbol{0},\lambda)(\bigoplus_{\lambda\in\Lambda}\Z^d)\rtimes\Lambda_\gamma$. For all $i\in\{1,\dots,d\}$, we denote by $t_i$ the element of $\bigoplus_{\lambda\in\Lambda}\Z^d$ such that $\supp(t_i)\subseteq\{e_\Lambda\}$ and $t_i(e_\Lambda)$ is the $i$-th element of the canonical basis of $\Z^d$. Let us also define $s_i\coloneqq (t_i,e_\Lambda)\in S$. Then the group $S$ is generated by $s_1,\dots,s_d$. Therefore, a coset $q\in\Gamma/\Gamma_\gamma$ is fixed by every element of $S$ if and only if it is fixed by all the elements $s_1,\dots,s_d$. So without loss of generality, it is sufficient to prove that the number of $q\in\Gamma/\Gamma_\gamma$ that are fixed by $s_1$ is $\geq (1-\varepsilon_\gamma)[\Gamma:\Gamma_\gamma]$.  For all  $q=(f,\lambda)\Gamma_\gamma\in \Gamma/\Gamma_\gamma$, we have 
\begin{align}\label{eq.equivalence}
s_1q=q &\Leftrightarrow (f,\lambda)\inv (t_1,e_\Lambda)(f,\lambda)\in \Gamma_\gamma\nonumber \\ &\Leftrightarrow (t_1(\lambda\cdot), e_\Lambda)\in \Gamma_\gamma\nonumber \\
&\Leftrightarrow \forall q'\in E, \sum_{\lambda'\in q'}t_1(\lambda\lambda')\in (p_\gamma\Z)^d.
\end{align}
This equivalence shows that for all $q\in\Gamma/\Gamma_\gamma$, we have $s_1q=q$ if and only if $s_1\pi(q)=\pi(q)$. Thus, we get that
\[\textstyle\{q\in \Gamma/\Gamma_\gamma\colon s_1q=q\}=[\bigoplus_{\lambda\in\Lambda}\Z^d:A_\gamma]\{q\in\Gamma/(\bigoplus_{\lambda\in\Lambda}\Z^d)\rtimes\Lambda_\gamma\colon s_1q=q\}.\]
Using the bijection $\phi$ defined in the proof of Claim \ref{claim.index},  the equivalence \eqref{eq.equivalence} implies that there is exactly $[\Lambda:\Lambda_\gamma]-l$ elements $q\in\Gamma/(\bigoplus_{\lambda\in\Lambda}\Z^d)\rtimes\Lambda_\gamma$ such that $s_1q=q$. We therefore obtain
\begin{align*}
\frac{\lvert\{q\in\Gamma/\Gamma_\gamma\colon s_1q=q\}\rvert}{[\Gamma:\Gamma_\gamma]}&=\frac{([\Lambda:\Lambda_\gamma]-l)[\bigoplus_\Lambda\Z^d:A_\gamma]}{[\Lambda:\Lambda_\gamma][\bigoplus_\Lambda\Z^d:A_\gamma]} \\
&>1-\varepsilon_\gamma,
\end{align*}  
which finishes the proof of the claim.
\end{cproof}

%\begin{cproof} A coset $q\in\Gamma/\Gamma_\gamma$ is fixed by every element of $S$ if and only if it is fixed by the elements $s_1,\dots,s_d\in S$ where $s_i$ is defined by $s_i\coloneqq (t_i,e_\Lambda)$ with $t_i(e_\Lambda)$ being the $i$-th element of the canonical basis of $\Z^d$. So without loss of generality, it is sufficient to prove that the number of $q\in\Gamma/\Gamma_\gamma$ that are fixed by $s_1$ is $\geq (1-\varepsilon_\gamma)[\Gamma:\Gamma_\gamma]$. Let $q\in \Gamma/\Gamma_\gamma$. Write $q=(f,\lambda)\Gamma_\gamma$ with $(f,\lambda)\in\Gamma$. Then 
%\begin{align*}
%s_1q=q &\Leftrightarrow (f,\lambda)\inv (t_1,e_\Lambda)(f,\lambda)\in \Gamma_\gamma \\ &\Leftrightarrow (t_1(\lambda\cdot), e_\Lambda)\in \Gamma_\gamma \\
%&\Leftrightarrow \forall q'\in E, \sum_{\lambda'\in q'}t_1(\lambda\lambda')\in (p_\gamma\Z)^d
%\end{align*}
%This happens if and only if all these sums are $0$, which exactly means that for all $q'\in E$, $\lambda\inv\notin q'$. There are exactly $[\Lambda:\Lambda_\gamma]-l$ such left cosets of $\Lambda_\gamma$ and therefore there are exactly $([\Lambda:\Lambda_\gamma]-l)[\bigoplus_\Lambda\Z^d:A_\gamma]$ left coset $q\in\Gamma/\Gamma_\gamma$ such that $s_1q=q$. We therefore obtain
%\begin{align*}
%\frac{\lvert\{q\in\Gamma/\Gamma_\gamma\colon s_1q=q\}\rvert}{[\Gamma:\Gamma_\gamma]}&=\frac{([\Lambda:\Lambda_\gamma]-l)[\bigoplus_\Lambda\Z^d:A_\gamma]}{[\Lambda:\Lambda_\gamma][\bigoplus_\Lambda\Z^d:A_\gamma]} \\
%&>1-\varepsilon_\gamma,
%\end{align*}  
%which finishes the proof of the claim.
%\end{cproof}
The sequence of finite-index subgroups $(\Gamma_\gamma)_{\gamma\in\Gamma\setminus\{e_\Gamma\}}$ satisfies all the assumptions of Theorem \ref{thm.allosterycriterion}, therefore $\Gamma=\Z^d\wr\Lambda$ is allosteric.
\end{proof}

\section{Proof of the main theorem
}\label{sec.mainthm}
We are now ready to prove Theorem \ref{thm.nonalmostfinite}.

\begin{proof}[Proof of Theorem \ref{thm.nonalmostfinite}]
Let $\Lambda$ be a finitely generated torsion free nilpotent group and let $d\in\N^*$. Let $\Gamma\coloneqq \Z^d\wr\Lambda$. By a result of Gruenberg \cite{Gruenberg}, $\Lambda$ is residually $p$-finite for every prime $p$. By Theorem \ref{thm.allosterywreath}, $\Gamma$ admits profinite allosteric actions. These actions are therefore topologically free, minimal actions on the Cantor space. Since they are not essentially free, they cannot be almost finite by Lemma \ref{lem.afnonessentiallyfree}. This proves the theorem. 
\end{proof}

\section*{Acknowledgment}
I thank David Kerr, Petr Naryshkin and Owen Tanner for a stimulating discussion on allostery and $\mathrm{C}^*$-algebras, as well as for their comments on a preliminary version of this work. I also thank the organizers of the conference \emph{``Group Actions: Dynamics, Measure, Topology''} in Münster during which this interaction took place. I thank Damien Gaboriau for very helpful comments on a first version of this paper. Finally, I thank the reviewer for many remarks and corrections and for pointing out a mistake in an earlier version of the proof of Claim 2.

\bibliographystyle{alpha}
\bibliography{biblio}

\end{document}

%% file: Commandes.tex
\newcommand{\N}{\mathbb N}
\newcommand{\Z}{\mathbb Z}

\newcommand{\supp}{\mathrm{supp}}

\newcommand{\defin}[1]{\textbf{\textit{#1}}}
\newcommand{\inv}{^{-1}}

\newcommand{\Stab}{\mathrm{Stab}}

\newcommand{\Sub}{\mathrm{Sub}}

\newcommand{\onto}{\twoheadrightarrow}

%% file: Amenable_wreath_products_with_non_almost_finite_actions_on_the_Cantor_set.bbl
\newcommand{\etalchar}[1]{$^{#1}$}
\begin{thebibliography}{GGK{\etalchar{+}}22}

\bibitem[AD87]{Anantharaman}
Claire Anantharaman-Delaroche.
\newblock Syst{\`e}mes dynamiques non commutatifs et moyennabilit{\'e}. ({Non}
  commutative dynamical systems and amenability).
\newblock {\em Math. Ann.}, 279(1-2):297--315, 1987.

\bibitem[AE07]{AbertElek}
Miklós Ab\'ert and Gábor Elek.
\newblock Non-abelian free groups admit non-essentially free actions on rooted
  trees.
\newblock arXiv:0707.0970, 2007.

\bibitem[AE12]{AbertElek2}
Miklós Ab\'ert and Gábor Elek.
\newblock {Hyperfinite actions on countable sets and probability measure
  spaces}.
\newblock In {\em Dynamical systems and group actions. Dedicated to Anatoli
  Stepin on the occasion of his 70th birthday}, pages 1--16. Providence, RI:
  American Mathematical Society (AMS), 2012.

\bibitem[AS94]{ArchboldSpielberg}
Robert~J. Archbold and John~S. Spielberg.
\newblock Topologically free actions and ideals in discrete
  {{\(C^*\)}}-dynamical systems.
\newblock {\em Proc. Edinb. Math. Soc., II. Ser.}, 37(1):119--124, 1994.

\bibitem[Bek20]{Bekka}
Bachir Bekka.
\newblock Character rigidity of simple algebraic groups.
\newblock arXiv:1908.06928, 2020.

\bibitem[BG04]{BergeronGaboriau}
Nicolas Bergeron and Damien Gaboriau.
\newblock {Asymptotique des nombres de Betti, invariants \(\ell^2\) et
  laminations}.
\newblock {\em {Comment. Math. Helv.}}, 79(2):362--395, 2004.

\bibitem[BLT19]{BeckerLubotzkyThom}
Oren {Becker}, Alex {Lubotzky}, and Andreas {Thom}.
\newblock {Stability and invariant random subgroups}.
\newblock {\em {Duke Math. J.}}, 168(12):2207--2234, 2019.

\bibitem[CET{\etalchar{+}}21]{classification}
Jorge Castillejos, Samuel Evington, Aaron Tikuisis, Stuart White, and Wilhelm
  Winter.
\newblock Nuclear dimension of simple {{\(\mathrm{C}^*\)}}-algebras.
\newblock {\em Invent. Math.}, 224(1):245--290, 2021.

\bibitem[CJK{\etalchar{+}}18]{CJKMSTD}
Clinton~T. Conley, Steve~C. Jackson, David Kerr, Andrew~S. Marks, Brandon
  Seward, and Robin~D. Tucker-Drob.
\newblock F{{\o}}lner tilings for actions of amenable groups.
\newblock {\em Math. Ann.}, 371(1-2):663--683, 2018.

\bibitem[CP17]{CreutzPeterson}
Darren Creutz and Jesse Peterson.
\newblock {Stabilizers of ergodic actions of lattices and commensurators}.
\newblock {\em {Trans. Am. Math. Soc.}}, 369(6):4119--4166, 2017.

\bibitem[Cre17]{Creutz}
Darren Creutz.
\newblock {Stabilizers of actions of lattices in products of groups}.
\newblock {\em {Ergodic Theory Dyn. Syst.}}, 37(4):1133--1186, 2017.

\bibitem[DM14]{DudkoMedynets}
Artem Dudko and Konstantin Medynets.
\newblock {Finite factor representations of Higman-Thompson groups.}
\newblock {\em {Groups Geom. Dyn.}}, 8(2):375--389, 2014.

\bibitem[EN17]{ElliottNiu}
George~A. Elliott and Zhuang Niu.
\newblock The {{\({C}^\ast\)}}-algebra of a minimal homeomorphism of zero mean
  dimension.
\newblock {\em Duke Math. J.}, 166(18):3569--3594, 2017.

\bibitem[GGK{\etalchar{+}}22]{PICP}
Eusebio Gardella, Shirly Geffen, Julian Kranz, Petr Naryshkin, and Andrea
  Vaccaro.
\newblock Tracially amenable actions and purely infinite crossed products.
\newblock arXiv:2211.16872, 2022.

\bibitem[GGKN22]{PICP2}
Eusebio Gardella, Shirly Geffen, Julian Kranz, and Petr Naryshkin.
\newblock Classifiability of crossed products by nonamenable groups.
\newblock arXiv:2201.03409, 2022.

\bibitem[GH09]{BooleanAlgebras}
Steven Givant and Paul Halmos.
\newblock {\em Introduction to {Boolean} algebras}.
\newblock Undergraduate Texts Math. New York, NY: Springer, 2009.

\bibitem[GK10]{GiolKerr}
Julien Giol and David Kerr.
\newblock Subshifts and perforation.
\newblock {\em J. Reine Angew. Math.}, 639:107--119, 2010.

\bibitem[GNS00]{GrigorchukNekrashevichSushchanskii}
Rostislav~I. Grigorchuk, Volodymyr~V. Nekrashevich, and Vitaly~I. Suschanskii.
\newblock {Automata, dynamical systems, and groups}.
\newblock In {\em Dynamical systems, automata, and infinite groups. Transl.
  from the Russian}, pages 128--203. Moscow: MAIK Nauka/Interperiodica
  Publishing, 2000.

\bibitem[Gri11]{Grigorchuk}
Rostislav~I. Grigorchuk.
\newblock {Some topics in the dynamics of group actions on rooted trees.}
\newblock {\em {Proc. Steklov Inst. Math.}}, 273:64--175, 2011.

\bibitem[Gru57]{Gruenberg}
Karl~W. Gruenberg.
\newblock {Residual properties of infinite soluble groups}.
\newblock {\em {Proc. Lond. Math. Soc. (3)}}, 7:29--62, 1957.

\bibitem[HW23]{HirshbergWu}
Ilan Hirshberg and Jianchao Wu.
\newblock Long thin covers and nuclear dimension.
\newblock arXiv:2308.12504, 2023.

\bibitem[Jos21]{josephContinuumAllosteric2021}
Matthieu Joseph.
\newblock Continuum of allosteric actions for non-amenable surface groups.
\newblock arXiv:2110.01068, to appear in Ergod. Theory Dyn. Syst., 2021.

\bibitem[Ker20]{Kerr}
David Kerr.
\newblock Dimension, comparison, and almost finiteness.
\newblock {\em J. Eur. Math. Soc. (JEMS)}, 22(11):3697--3745, 2020.

\bibitem[KN21]{KerrNaryshkin}
David Kerr and Petr Naryshkin.
\newblock Elementary amenability and almost finiteness.
\newblock arXiv:2107.05273, 2021.

\bibitem[Nar23]{naryshkin2023group}
Petr Naryshkin.
\newblock Group extensions preserve almost finiteness.
\newblock arXiv:2304.02456, 2023.

\bibitem[Niu19]{Niu}
Zhuang Niu.
\newblock Comparison radius and mean topological dimension:
  $\mathbb{Z}^d$-actions.
\newblock arXiv:1906.09171, 2019.

\bibitem[OS22]{OrtegaScarparo}
Eduard Ortega and Eduardo Scarparo.
\newblock Almost finiteness and homology of certain non-free actions.
\newblock {\em Groups Geom. Dyn.}, 2022.

\bibitem[PT16]{PetersonThom}
Jesse Peterson and Andreas Thom.
\newblock {Character rigidity for special linear groups.}
\newblock {\em {J. Reine Angew. Math.}}, 716:207--228, 2016.

\bibitem[SZ94]{StuckZimmer}
Garrett~J. Stuck and Robert~J. Zimmer.
\newblock {Stabilizers for ergodic actions of higher rank semisimple groups}.
\newblock {\em {Ann. Math. (2)}}, 139(3):723--747, 1994.

\bibitem[Tu99]{Tu}
Jean-Louis Tu.
\newblock The {Baum}-{Connes} conjecture for amenable foliations.
\newblock {\em \(K\)-Theory}, 17(3):215--264, 1999.

\bibitem[TW13]{TW}
Andrew~S. Toms and Wilhelm Winter.
\newblock Minimal dynamics and {{\(K\)}}-theoretic rigidity: {Elliott}'s
  conjecture.
\newblock {\em Geom. Funct. Anal.}, 23(1):467--481, 2013.

\bibitem[Ver12]{Vershik}
Anatoly~M. Vershik.
\newblock {Totally nonfree actions and the infinite symmetric group}.
\newblock {\em {Mosc. Math. J.}}, 12(1):193--212, 2012.

\bibitem[Zhe19]{Zheng}
Tianyi Zheng.
\newblock On rigid stabilizers and invariant random subgroups of groups of
  homeomorphisms.
\newblock arXiv:1901.04428, 2019.

\end{thebibliography}
